\documentclass[12pt]{article}

\usepackage{mathtools,amsthm,amssymb,amsmath}
\usepackage{graphicx}
\usepackage{hyperref}
\usepackage{fullpage}

\theoremstyle{plain}
\newtheorem{theorem}{Theorem}

\theoremstyle{definition}
\newtheorem{definition}[theorem]{Definition}
\newtheorem{example}[theorem]{Example}

\theoremstyle{remark}
\newtheorem{remark}[theorem]{Remark}

\author{Jean-Paul Allouche \\
CNRS, IMJ-PRG \\
Sorbonne, 4 Place Jussieu \\
F-75252 Paris Cedex 05 \\
France \\
{\tt jean-paul.allouche@imj-prg.fr}\\
\and
Jeffrey Shallit \\
School of Computer Science \\
University of Waterloo \\
Waterloo, Ontario  N2L 3G1 \\
Canada \\
{\tt shallit@uwaterloo.ca}
}

\title{Automatic sequences are also non-uniformly morphic}

\begin{document}

\maketitle

\begin{abstract}
It is well-known that there exist
infinite sequences that are the fixed point of non-uniform morphisms,
but not $k$-automatic for any $k$.  In this note we show that 
every $k$-automatic sequence is the image of a fixed point of a
{\it non-uniform\/} morphism.
\end{abstract}

\section{Introduction, Definitions, Notation}

{\it Combinatorics on words} deals with ``alphabets'', ``words'', ``languages'', and ``morphisms of monoids''.  
The first three notions are inspired by the usual meaning of these words in English.  Below we recall the precise definitions.

\begin{definition}
A finite set ${\cal A}$ is called an {\em alphabet}. A {\em word} over the alphabet ${\cal A}$
is a finite (possibly empty) sequence of symbols from ${\cal A}$. We let ${\cal A}^*$ denote 
the set of all words on ${\cal A}$. A subset of ${\cal A}^*$ is called a {\em language} on ${\cal A}$. 
The {\em length} of a word $w$, denoted $|w|$, is the number of symbols that it contains (the length 
of the empty word $\epsilon$ is $0$). The {\em concatenation} of two words $w = a_1 a_2 \cdots a_r$ 
and $z = b_1 b_2 \cdots b_s$ of lengths $r$ and $s$, respectively, is the word 
denoted $wz$ defined 
by $wz = a_1 a_2 \cdots a_r b_1 b_2 \cdots b_s$ of length $r+s$ obtained by gluing $w$ and $z$ 
in order. The set ${\cal A}^*$ equipped with concatenation is called the {\em free monoid} generated 
by ${\cal A}$. The concatenation of a word $w = a_1 a_2 \cdots a_r$ and a sequence 
$(x_n)_{n \geq 0}$ is the sequence $a_1 a_2 \cdots a_r x_0 x_1 \cdots$, 
denoted $w \ (x_n)_{n \geq 0}$. A word $w$ is called a {\em prefix} of the word $z$ (or of the infinite 
sequence $(x_n)_{n \geq 0}$) if there exists a word $y$ with $z = wy$ (respectively a sequence 
$(y_n)_{n \geq 0}$ with $(x_n)_{n \geq 0} = w \ (y_n)_{n \geq 0}$). 

Let $(u_{\ell})_{\ell \geq 0}$ be a sequence of words in of ${\cal A}^*$, and $(a_n)_{n \geq 0}$ 
be a sequence over the alphabet ${\cal A}$. The sequence $(u_{\ell})_{\ell \geq 0}$ is said to {\em converge}
to the sequence $(a_n)_{n \geq 0}$ if the length of the largest prefix of $u_{\ell}$ that is also a 
prefix of $(a_n)_{n \geq 0}$ tends to infinity with $\ell$.
\end{definition}

\begin{remark}
It is straightforward that ${\cal A}^*$ equipped with concatenation is indeed
a monoid: concatenation is associative, and the empty word $\epsilon$ is the identity
element. This monoid is {\em free}; intuitively,
this means that there are no
relations between elements, other than the relations arising from
the associative property
and the fact that the empty word is the identity element. In particular, this
monoid is not {\em commutative} if ${\cal A}$ has at least two distinct
elements.
\end{remark}

\begin{definition}
Let ${\cal A}$ and ${\cal B}$ be two alphabets.
A {\em morphism} from ${\cal A}^*$ 
to ${\cal B}^*$ is a map $\varphi$ from
${\cal A}^*$ to ${\cal B}^*$ such that, 
for all words $u$ and $v$, one has $\varphi(uv) = \varphi(u)\varphi(v)$.
A morphism of ${\cal A}^*$ is a morphism from ${\cal A}^*$ to itself.

If there exists a positive integer $k$ such that $\varphi(a)$ has length $k \geq 1$
for all $a \in {\cal A}$, the morphism $\varphi$ is said to be {\em $k$-uniform}.
If a morphism is $k$-uniform for some $k \geq 1$, it is called a {\em uniform morphism}.  Otherwise it is {\em non-uniform}.  A $1$-uniform morphism is
sometimes called a {\em coding}.

\begin{example}
The {\it Thue-Morse morphism} $\mu$ sending $0 \rightarrow 01$ and
$1 \rightarrow 10$ is $2$-uniform.   In contrast, the {\it
Fibonacci morphism}
$\tau$ sending $a$ to $ab$ and $b$ to $a$ is non-uniform.
\end{example}

\end{definition}

\begin{remark}
A morphism $\varphi$ from ${\cal A}^*$ to ${\cal B}^*$ is 
completely determined by the 
values of $\varphi(a)$ for $a \in {\cal A}$. Namely, if the word $u$ is equal 
to $a_1 a_2 \cdots a_n$ with $a_j \in {\cal A}$, then $\varphi(u) = 
\varphi(a_1) \varphi(a_2) \cdots \varphi(a_n)$.
\end{remark}

\begin{definition}
An infinite sequence $(a_n)_{n \geq 0}$ taking values in the alphabet 
${\cal A}$ is said to be {\em pure morphic} if there exist a morphism
$\varphi$ of ${\cal A}^*$ and a word $x \in {\cal A}^*$ such that
\begin{itemize}
\item the word $\varphi(a_0)$ begins with $a_0$; i.e.,
        there exists a word $x$ such that $\varphi(a_0) = a_0 x$;
\item iterating $\varphi$ starting from $x$ never gives the empty word, i.e., 
         for each integer $\ell$, $\varphi^{\ell}(x) \neq \epsilon$;
\item the sequence of words $(\varphi^{\ell}(a_0))_{\ell \geq 0}$ converges to the
         sequence $(a_n)_{n \geq 0}$ when $\ell \to \infty$.
\end{itemize}
\end{definition}

\begin{remark}
It is immediate that 
\begin{align*}
\varphi(a_0)   &= a_0 x \\
\varphi^2(a_0) &= \varphi(\varphi(a_0)) = \varphi(a_0 x) = \varphi(a_0) \varphi(x)
= a_0 x \varphi(x) \\
\varphi^3(a_0) &= \varphi(\varphi^2(a_0)) = \varphi(a_0 x \varphi(x)) =
\varphi(a_0) \varphi(x) \varphi^2(x) = a_0 x \varphi(x) \varphi^2(x) 
\end{align*}
and more generally
$$ \varphi^{\ell}(a_0) = a_0 x \varphi(x) \varphi^2(x) \cdots \varphi^{\ell-1}(x) $$
for all $\ell \geq 0$.
\end{remark}

\begin{definition}
An infinite sequence $(a_n)_{n \geq 0}$ taking values in
${\cal A}$ is said to be {\em morphic} if there exist an alphabet 
${\cal B}$ and an infinite sequence $(b_n)_{n \geq 0}$ over the alphabet
${\cal B}$
such that
\begin{itemize}
\item the sequence $(b_n)_{n \geq 0}$ is pure morphic;
\item there exists a coding  from ${\cal B}^*$ to
      ${\cal A}^*$ sending the sequence $(b_n)_{n \geq 0}$ to
      the sequence $(a_n)_{n \geq 0}$; i.e., the sequence 
      $(a_n)_{n \geq 0}$ is the pointwise image of $(b_n)_{n \geq 0}$.
\end{itemize}
If the morphism making $(b_n)_{n \geq 0}$ morphic is $k$-uniform,
then the sequence $(a_n)_{n \geq 0}$  is said to be {\em $k$-automatic}.
The word ``automatic'' comes from the fact that the sequence $(a_n)_{n \geq 0}$
can be generated by a finite automaton (see \cite{AS} for more details on this topic).
\end{definition}

\begin{remark}
A morphism $\varphi$ of ${\cal A}^*$ can be extended to infinite sequences
with values in ${\cal A}$ by defining
$$\varphi((a_n)_{n \geq 0}) = \varphi(a_0 a_1 a_2 \cdots) := 
\varphi(a_0) \varphi(a_1) \varphi(a_2) \cdots .$$
It is easy to see that a pure morphic sequence is a {\em fixed point} of 
(the extension to infinite sequences of) some morphism: actually, with the
notation above, it is {\em the} fixed point of $\varphi$ beginning with 
$a_0$. A pure morphic sequence is also called an {\em iterative fixed point} 
of some morphism (because of the construction of that fixed point), while a 
morphic sequence is the pointwise image of an iterative fixed point of some 
morphism, and a $k$-automatic sequence is the pointwise image of the iterative 
fixed point of a $k$-uniform morphism.
\end{remark}

\section{The main result}

Looking at the definitions above,
we see that every automatic sequence is also a
morphic sequence.
We will prove that every automatic sequence can be obtained as
a morphic sequence where the involved morphism is {\em not uniform}.
 
\begin{definition}
We say a sequence is {\em non-uniformly pure morphic} if it is the iterative
fixed point of a non-uniform morphism. We say that a sequence is
{\em non-uniformly morphic} if it is the image (under a coding)
of a non-uniformly pure morphic sequence. 
\end{definition}
For example, the sequence $abaababa \cdots$ generated by iterating the morphism 
$\tau$ defined above is non-uniformly pure morphic. This sequence 
is known as the {\em (binary) Fibonacci sequence}, since it is also equal to the limit 
of the sequence of words $(u_n)_{n \geq 0}$ defined by $u_0 := a$, 
$u_1 := ab$, $u_{n+2} := u_{n+1} u_n$ for each $n \geq 0$.

In order to avoid triviality, we certainly assume
(as M. Mend\`es France once pointed out to us)
that the alphabet of the 
non-uniform morphism involved in the above definition is
the same as the minimal 
alphabet of its fixed point. For example, the
fact that the morphism $0 \to 01$, $1 \to 10$, 
$2 \to 1101$, whose iterative fixed point beginning with $0$ is
also the iterative fixed point, beginning with $0$,
of the morphism $\mu$ --- namely the Thue-Morse sequence ---
does not make that sequence non-uniformly morphic.)

Although most non-uniformly morphic sequences are not automatic (e.g.,
the binary Fibonacci sequence is not automatic), some sequences can be 
simultaneously automatic and non-uniformly morphic. An example is the 
sequence ${\cal Z}$ formed by the lengths of the blocks of $1$'s between 
two consecutive zeros in the Thue-Morse sequence.
$$
\begin{array}{lll}
           && 0 \ 1 \ 1 \ 0 \ 1 \ 0 \ 0 \ 1 \ 1 \ 0 \ 0 \ 1 \ 0 \ 
                1 \ 1 \ 0 \ \cdots \\
           && 0 \ (1 1) \ 0 \ (1) \ 0 \ ( \ ) \  0 \ (1 1) \ 0 \ 
               ( \ ) \ 0 \ (1) \ 0 \ (1 1) \ 0 \ \cdots \\ 
{\cal Z}   &=& 2 \ 1 \ 0 \ 2 \ 0 \ 1 \ 2 \ \cdots \\
\end{array}
$$
As is well known \cite{Berstel}, this sequence is both the fixed point of the 
map sending $2 \rightarrow 210$, $1 \rightarrow 20$, and $0 \rightarrow 1$,
and also the image, under the coding $0 \rightarrow 2$, $1 \rightarrow 1$, 
$2 \rightarrow 0$, $3 \rightarrow 1$ of the fixed point of the map $0 \rightarrow 01$, 
$1 \rightarrow 20$, $2 \rightarrow 23$, and $3 \rightarrow 02$.

In view of this example, one can ask which non-uniformly morphic sequences are also 
$k$-automatic for some integer $k \geq 2$, or which automatic sequences are also 
non-uniformly morphic. We prove here that {\em all\/} automatic sequences are
also non-uniformly morphic.

\begin{theorem}
\label{newAutomatic}
Let $(a_n)_{n \geq 0}$ be an automatic sequence taking
values in the alphabet ${\cal A}$. Then $(a_n)_{n \geq 0}$
is {\em also} non-uniformly morphic. Furthermore, if
$(a_n)_{n \geq 0}$ is the iterative fixed point of a uniform
morphism, then there exist an alphabet ${\cal B}$ of 
cardinality $(3 + \#{\cal A})$ and a sequence $(a'_n)_{n \geq 0}$ 
with values in ${\cal B}$, such that $(a'_n)_{n \geq 0}$ is the
iterative fixed point of some non-uniform morphism with domain ${\cal B}^*$
and $(a_n)_{n \geq 0}$ is the image of $(a'_n)_{n \geq 0}$ under a 
coding.
\end{theorem}

\begin{proof}
We start with the first assertion. First, we may suppose that the first letter of $(a_n)_{n \geq 0}$ 
is different from all $a_j$ for $j \geq 1$. If not, take a letter $\alpha$ not in ${\cal A}$ and consider
the sequence $\alpha a_1 a_2 \cdots$. This sequence is automatic and the morphism 
$\alpha \to a_0$ and $a \to a$ for all letters $a$ in ${\cal A}$ sends it to $(a_n)_{n \geq 0}$. 

We may also suppose that the sequence $(a_n)_{n \geq 0}$ is not ultimately periodic (otherwise 
the result is trivial: if $u$ and $v$ are two words over the
alphabet ${\cal A}$, the sequence $u v v v \cdots$ is 
the iterative fixed point of the morphism $\alpha \to u$ and $a \to v^j$ for all $a \in {\cal A}$, where 
$j$ is chosen so that $j|v| \neq |u|$).

Thus we now start with an automatic non-ultimately periodic 
sequence, still called $(a_n)_{n \geq 0}$, with $a_0 = \alpha \neq a_1$. Since the sequence 
$(a_n)_{n \geq 0}$ is the pointwise image of the iterative fixed point $(x_n)_{n \geq 0}$ of some 
uniform morphism, we may suppose, by replacing $(a_n)_{n \geq 0}$ with $(x_n)_{n \geq 0}$, that 
$(a_n)_{n \geq 0}$ itself is the iterative fixed point beginning with $a_0 = \alpha \neq a_j$ for all
$j \geq 1$ of a uniform morphism $\gamma$ with domain ${\cal A}^*$, and still non-ultimately periodic.

We claim that there exists a $2$-letter word $bc$ such that $\gamma(bc)$ contains $bc$ as a 
factor. Namely, since $\gamma$ is uniform, it has exponential growth
(that is, iterating 
$\gamma$ on each letter gives words of exponentially growing length).
Hence there
exists a letter $b$ that is {\em expanding}; i.e., such that some power of
$\gamma$ maps $b$ to a word that contains at least two occurrences of $b$ 
(see, e.g., \cite{Salomaa}). By replacing $\gamma$ with this power of $\gamma$, 
we can write $\gamma(b) = ubvbw$ for some words $u, v, w$. By replacing this
new $\gamma$ with $\gamma^2$,
we can also suppose that both $u$ and $w$ are nonempty. 
Let $c$ be the letter following the prefix $ub$ of $ubvbw$.  Now there
are two cases:
\begin{itemize}
\item if $c \neq b$, then $v = cy$ for some word $y$, and $\gamma(b) = ubcyw$,
and $\gamma(bc) = \gamma(b)\gamma(c)$ 
contains $bc$ as a factor;
\item if $c=b$, then $\gamma(b) = ubbz$ for some word $z$, and 
$\gamma(bb) = ubbzubbz$ contains $bb$ as a factor.
\end{itemize}
In both cases, there exist two 
letters $b$ and $c$, not necessarily distinct,
such that $\gamma(b) = w_1 b c w_2$ and $\gamma(bc) = w_1 b c w_3$,
where $w_1, w_2$ are non-empty words. Note, in particular,
that $b$ can be chosen distinct from $a_0$ ($w_1$ is non-empty and $a_0 = \alpha$ is different from all $a_j$ for
$j \geq 1$).

Now define a new alphabet ${\cal A}' := {\cal A} \cup \{b',c'\}$, where $b', c'$ are 
two new letters not in ${\cal A}$. Define the morphism $\gamma'$ with
domain ${\cal A}'$ as follows: 
if the letter $y$ belongs to ${\cal A} \setminus \{b\}$,
then $\gamma'(y) := \gamma(y)$.
If $y = b$, define $\gamma'(b) := w_1 b' c' w_2$. Finally, define $\gamma'(b')$ and
$\gamma'(c')$ as follows: first recall that $\gamma(bc) = w_1 b c w_3$; cut the word
$w_1 b c w_3$ into (any) two non-empty words of unequal length, say $w_1 b c w_3 := zt$, 
and define $\gamma'(b') := z$, $\gamma'(c') := t$. 

By construction, $\gamma'$
is not uniform. Its iterative fixed point beginning with $a_0$ clearly exists,
and we denote it by
$(a'_n)_{n \geq 0}$. This sequence has the property that each $b'$ in it is followed by a $c'$
and each $c'$ is preceded by a $b'$. 
We let $D$ denote the coding that sends each letter of ${\cal A}$ to itself, and 
sends $b'$ to $b$ and $c'$ to $c$. For every letter $x$ belonging to ${\cal A}' \setminus \{b, \, b', c'\}$ 
we have $\gamma(x) = \gamma'(x)$. Hence $D \circ \gamma'(x) = D \circ \gamma(x) = \gamma(x)$.
For $x = b$, we have $D \circ \gamma'(b) = D(w_1 b' c' w_2) = w_1 b c w_2 = \gamma(b)$. 
Furthermore, we have $D \circ \gamma'(b'c') = D(zt) = zt = w_1bcw_3 = \gamma(bc)$.

Now let $P_k$ be the prefix of the sequence $(a'_n)_{n \geq 0}$ that ends with $c'$ and contains 
exactly $k$ occurrences of the letter $c'$. Each occurrence of $c'$ must be preceded by a $b'$, so that
$P_k$ can be written $P_k = p_1 b'c' p_2 b'c' \cdots p_k b'c'$ where the $p_i$'s are words over
the alphabet ${\cal A}$. We have
\begin{align*}
D \circ \gamma' (P_k) &= D \circ \gamma' (p_1 b'c' p_2 b'c' \cdots p_k b'c') \\
&= (D \circ \gamma'(p_1)) (D \circ \gamma'(b'c')) (D \circ \gamma'(p_2)) (D \circ \gamma'(b'c'))
\cdots (D \circ \gamma'(p_k)) (D \circ \gamma' (b'c')) \\
&= \gamma(p_1) \gamma(bc) \gamma(p_2) \gamma(bc) \cdots \gamma(p_k) \gamma (bc) \\
&= \gamma(p_1 bc p_2 bc \cdots p_k bc) \\
&= \gamma \circ D (p_1 b'c' p_2 b'c' \cdots p_k b'c')  \\
&= \gamma \circ D(P_k) .
\end{align*}
Letting $k$ go to infinity, we obtain that
$D \circ \gamma'((a'_n)_{n \geq 0}) = \gamma \circ D((a'_n)_{n \geq 0})$, but 
$\gamma'((a'_n)_{n \geq 0}) = (a'_n)_{n \geq 0}$, so that
$D((a'_n)_{n \geq 0}) = \gamma \circ D((a'_n)_{n \geq 0})$.
Hence $D((a'_n)_{n \geq 0})$ is the iterative fixed point of $\gamma$ 
beginning with $a_0$. Hence 
it is equal to the sequence $(a_n)_{n \geq 0}$.

The second assertion is a consequence of the fact that we introduced at most only three
new letters $\alpha, b', c'$ in the proof above.  
\end{proof}


\begin{thebibliography}{99.}%

\bibitem{ubi} J.-P. Allouche and J. Shallit, The ubiquitous Prouhet-Thue-Morse 
sequence, in \textit{Sequences and Their Applications},
Proceedings of SETA'98, C. Ding, T. Helleseth and H. Niederreiter (Eds.),
(Springer, 1999), pp.~1--16.
 
\bibitem{AS} J.-P. Allouche and J. Shallit,  \textit{Automatic Sequences: Theory, 
Applications, Generalizations}, (Cambridge University Press, 2003).

\bibitem{Berstel} J. Berstel, Sur la construction de mots sans carr\'e,
S\'em. Th\'eor. Nombres Bordeaux, 1978--1979, Expos\'e 18, 18-01--18-15.

\bibitem{Salomaa} A. Salomaa, On exponential growth in Lindenmayer systems,
Nederl. Akad. Wetensch. Proc. Ser. A {\bf 76} (= Indag. Math. {\bf 35} (1973)),
23--30.


\end{thebibliography}
\end{document}